\documentclass[12pt,reqno,sumlimits]{amsart}

\usepackage{amssymb,amscd,amsmath,epsfig}
\newtheorem{thm}{Theorem}[section]

\newtheorem{lem}[thm]{Lemma}
\newtheorem{prop}[thm]{Proposition}
\theoremstyle{definition}
\newtheorem{defn}{Definition}[section]
\newtheorem{rem}{Remark}[section]

\newcommand{\R}{{\mathbb R}}

\newcommand{\Q}{{\mathbb Q}}
\newcommand{\C}{{\mathbb C}}

\newcommand{\Z}{{\mathbb Z}}

\newcommand{\CP}{{\mathbb C}{\mathbb P}}

\newcommand{\calA}{{\mathcal A}}
\newcommand{\calB}{{\mathcal B}}
\newcommand{\calC}{{\mathcal C}}
\newcommand{\calD}{{\mathcal D}}
\newcommand{\calE}{{\mathcal E}}

\newcommand{\calG}{{\mathcal G}}
\newcommand{\calH}{{\mathcal H}}

\newcommand{\calM}{{\mathcal M}}

\newcommand{\calL}{{\mathcal L}}

\renewcommand{\to}{\longrightarrow}

\newcommand{\Aut}{\operatorname{Aut}}

\newcommand{\End}{\operatorname{End}}
\newcommand{\ev}{\operatorname{ev}}

\newcommand{\Tr}{\operatorname{Tr}}
\newcommand{\tr}{\operatorname{tr}}

\newsavebox{\savepar}

\numberwithin{equation}{section}

\newcounter{labelflag} 
\newcommand{\labelon}{\setcounter{labelflag}{1}}
\newcommand{\Label}[1]{
                       \ifnum\thelabelflag=1
                          \ifmmode
                             \makebox[0in][l]{\qquad\fbox{\rm#1}}
                          \else
                             \marginpar{\vspace{0.7\baselineskip}
                                        \hspace{-1.1\textwidth}
                                        \fbox{\rm#1}}
                          \fi
                       \fi
                       \label{#1}
                      }
\labelon
\usepackage[top=3cm, bottom=1.8cm, left=2cm, right=2cm,headsep=30pt]{geometry}

\newcommand{\BbP}{{\mathbb P}}

 \newcommand{\pdo}{\Psi{\rm DO}}

 \newcommand{\dvol}{{\rm dvol}}

 \newcommand{\dg}{\dot\gamma}

 \newcommand{\ints}{\int_{S^1}}

 \newcommand{\dir}{\partial\kern-.570em /}
 \newcommand{\dire}{\partial\kern-.570em /{}^{\rm eq}}

 \newcommand{\pdoz}{\pdo_{\leq 0}}

 \newcommand{\wgti}{(1+\Delta)^{-1}}

 \newcommand{\wgts}{(1+\Delta)^s}
 \newcommand{\ipo}[2]{\langle {#1},{#2}\rangle}

 \newcommand{\diff}{{\rm Diff} }
 
 \newcommand{\kk}{2k-1 }
 
 \newcommand{\resw}{{\rm res}^{ W}}
 
 \newcommand{\be}{\overline\eta}
  
  \newcommand{\bxi}{\overline\xi}
\newcommand{\br}{\overline R}
\newcommand{\bmk}{\overline{ M}_k}
\newcommand{\bg}{\overline g}
\newcommand{\bx}{\overline X}
\newcommand{\by}{\overline Y}

\newcommand{\xl}{X^L}
\newcommand{\yl}{Y^L}
\newcommand{\zl}{Z^L}
\newcommand{\la}{\langle}
\newcommand{\ra}{\rangle}

\newcommand{\vol}{{\rm vol}}

\newcommand{\mabk}{\overline {M}_{k(a,b)} }

\newcommand{\bka}{\overline{M}_{k\vec a}}

\newcommand{\mapsnm}{{\rm Maps}(N,M)}
\newcommand{\pdoos}{\Psi{\rm DO}^*_0}
\newcommand{\Auto}{\rm Aut}
\newcommand{\Endo}{\rm End}

\newcommand{\cklo}{c_k^{\rm lo}}

\newcommand{\mmaps}{{\rm Maps}}

\newcommand{\nlm}{\nabla^{LM}}

\newcommand{\nm}{\nabla^{M}}

\newcommand{\wog}{\widetilde\Omega_{\mathfrak u}}

\newcommand{\nee}{\nabla^E}

\newcommand{\woeg}{\widetilde\Omega^\calE_{\mathfrak u}}

\newcommand{\che}{{\rm ch}}

\newcommand{\wnen}{\widetilde {\nabla}^{\nabla}}

\newcommand{\calas}{\calA^*}
\newcommand{\frakg}{\mathfrak g}
\newcommand{\adP}{{\rm Ad}\ P}

\newcommand{\keras}{{\rm ker}\ d_A^*}

\newcommand{\PD}{{\rm PD}}

\newcommand{\fraku}{{\mathfrak u}}

\newcommand{\mok}{\calM_{0,k}(A)}

\newcommand{\mokm} {\calM_{0,k-1}(A)}

\newcommand{\pdg}{{\rm PD}(\gamma)}
\newcommand{\wn}{\widetilde \nabla}

\begin{document}
\title{Traces and characteristic classes in infinite dimensions}
\author[Y. Maeda]{Yoshiaki Maeda}
\address{Department of Mathematics\\
Keio University}
\email{maeda@math.keio.ac.jp}
\author[S. Rosenberg]{Steven Rosenberg}
\address{Department of Mathematics and Statistics\\
  Boston University}
\email{sr@math.bu.edu}
\maketitle

\section{Introduction}

\noindent {\it Note:  Parts of \S2.3 are not correct.  The equivariant curvatures  $\widetilde \Omega_{\fraku}$ and
$\widetilde \Omega^E_{\fraku}$  introduced below (2.4) take values in first order differential operators, by 
the work of T. McCauley \cite{tmc}.  Thus the leading order trace of powers of these operators is not defined.  In particular, we do not have a definition of the leading order equivariant Chern character or $\hat A$-class.  Thus Theorem 2.4(ii) and Theorem 2.5 are not correct.}

\bigskip

Infinite rank vector bundles often appear in mathematics and mathematical physics. As one example, 
 tangent bundles to
  spaces of maps $\mapsnm$ from one manifold to another are important in string theory and in formal proofs of the Aityah-Singer theorem on loop spaces.   In addition, gauge theories use bundles 
 associated to the basic fibration $\calA\to \calA/\calG$ of connections to connections modulo gauge 
 transformations.   Usually one focuses on finite dimensional associated moduli spaces to produce Gromov-Witten invariants and Donaldson/Seiberg-Witten invariants.  In this paper, we discuss the construction of characteristic classes directly on these infinite rank bundles and their applications to topology. The main results are the construction of a universal $\hat A$-polynomial and Chern character that control the 
 $S^1$-index theorem for all circle actions on a fixed vector bundle over a manifold (Thms.~\ref{newthm}, \ref{bigtwo}), and the detection of elements of infinite order in the diffeomorphism groups of 5-manifolds associated to projective algebraic K\"ahler surfaces (Thms.~\ref{bigthm}, \ref{lastthm}).
 
 These characteristic classes are modeled on Chern classes and Chern-Simons classes for complex vector bundles, but with the structure group $U(n)$ replaced by a gauge group $\calG = \Auto(E)$ or a larger group $\pdoos$ of zeroth order invertible pseudodifferential operators ($\pdo$s) acting on sections of a bundle $E$ over a closed
  manifold $M$. 
 Since finite rank Chern classes depend essentially on the ordinary matrix trace on $\mathfrak u(n)$, it
 is  natural to look for traces on the Lie algebra of $\pdoos.$  These traces come in two types:  one is built from the leading order symbol of a $\pdo$, and in  the gauge group case is just 
 $\int_M \tr (A)\ \dvol$ for $A\in \Gamma(\Endo(E)) = {\rm Lie}(\calG).$   The second is built from the Wodzicki residue of a $\pdo$.  These traces are quite different, in that the Wodzicki trace vanishes on $\Gamma(\Endo(E))$, but they share the crucial locality property that they are both integrals of pointwise computed functions.  Thus characteristic classes built from these traces are in theory as computable as finite rank Chern classes, a distinct advantage over  the usual operator trace.
 
 The leading order trace is fairly easy to work with.  For example, the tangent bundle to the loop space $LM$ is the sheaf-theoretic pushdown of $\ev^*TM$ for the evaluation map $\ev:LM\times S^1 \to M$, and the leading order Pontrjagin classes of $TLM$ are related to the Pontrjagin classes of $M$.  In particular, these leading order classes are often nonzero.  We can use these classes to restate the $S^1$-index theorem as a statement on $LM$ and to construct an equivariant universal $\hat A$-polynomial on $LM\times \calB$, with $\calB$ the space of metrics on $M$, which appears in the $S^1$-index theorem for every action on $M$.  We extend this to twisted Dirac operators by constructing a universal Chern character.
 %As the original motivation, we had hoped to make physics proofs of the index theorem on loop space rigorous, but the formal localization techniques seem out of reach.

 In contrast, the Wodzicki version of characteristic classes seems to be unrelated to the finite dimensional theory.
 While the Wodzicki-Pontrjagin classes vanish for $TLM$ and conjecturally on all $\pdoos$-bundles, the associated secondary/Chern-Simons classes are sometimes nonzero.  These WCS classes on $TLM$ can detect nontrivial elements in  $\pi_1(\diff(\bmk))$ for many Sasakian $5$-manifolds $\bmk$, $k\in \Z\setminus\{0\}.$   These manifolds are the total spaces of circle bundles over projective algebraic K\"ahler surfaces $M$, and come in  infinite families for each such $M$.
 
  In \S2, we discuss leading order classes, and in \S3 we discuss the Wodzicki classes.  One common theme is the use of $S^1$ actions $a:S^1\times M\to M.$ on compact manifolds.  Any action gives rise to both a map $a^L:M\to LM$, $a^L(m)(\theta) = a(m,\theta)$, taking a point to its orbit, and a map $a^D: S^1\to\diff(M)$ given by $a^D(\theta)(m)
  = a(\theta, m).$  This is just the set theory equality ${\rm Maps}(X\times Y,Z) = {\rm Maps}(X, {\rm Maps}(Y,Z)) = {\rm Maps}(Y, {\rm Maps}(X,Z))$ for $X = S^1, Y = Z = M.$
  
  We use $a^L$ in \S2 to discuss the $S^1$-index theorem.  To state the main result Thm.~\ref{bigtwo}, let $\calB$ be the space of Riemannian metrics on a spin manifold $M$, and let $\calC$ be the space of pairs $(\nabla, h)$, where $\nabla$ is a connection on a fixed complex bundle $E\to M$ and $h$ is a compatible hermitian metric on $E$.  
  Then there is a ``universal index form" $U$  on $LM\times\calB\times\calC$ such that for 
  each $S^1$ action $a$ on $(E,\nabla, h)\to M$ and Riemannian metric $g$ on $M$ for which the action is via isometries, there is an embedding $j = j_{(a,g,\nabla,h)}:M\to LM\times\calB\times\calC$ such that the $S^1$-index of the twisted Dirac operator is given by
  ${\rm ind}_{S^1} \dir_{\nee}=\int^{S^1}_{j_*[M]} U.$  In \S3, we use
the relationship between $a^L$ and $a^D$ and some K\"ahler geometry to sketch the results on $\pi_1(\diff(\bmk)).$ 
 
 We would like to think that this work touches on several topics that appeared in Prof.~Kobyashi's work: transformation groups (although only $S^1$ actions for us), and the interplay of Riemannian and complex geometry.  
 \bigskip
 
 We were privileged to have known Prof. Kobayashi for many years.  The second author was a graduate student at Berkeley when Prof. Kobayashi was department chair.  At that time, the math department was in a turf war with another department over office space. Although graduate students in a large department had little direct contact with the chair, letters between Prof. Kobayashi and the administration were regularly posted in the mailroom.  In contrast to the typical American style of aggressively defending our territory against intruders, Prof. Kobayashi's letters said in so many words that he would like to give offices to the other department but regretfully could not.   The reasons preventing the handover were always very complicated.  This tactic seemed to confound the administration, whose puzzled replies took longer and longer to appear in the mailroom and finally ceased altogether.  Already from this first encounter, which only involved reading letters, Prof. Kobayashi's gentle determination and sly humor were apparent.  Twenty years later, it was a great pleasure to re-encounter Prof. Kobayashi in Japan and to see that his mathematical mind and personality were unchanged.
 
 \section{Leading order classes and  applications}
 
 \subsection{Infinite rank bundles}

Any discussion of infinite rank bundles involves some initial technicalities, just because infinite dimensional vector spaces have many inequivalent  norm topologies.  In particular, the topologies on smooth functions on a compact manifold associated to different Sobolev norms are inequivalent.  

Thus we first have to decide which vector space to use as the model for the fiber of  an infinite rank vector bundle $\calE\to \calM$ over a paracompact base.  Based on the examples in the introduction, we choose fibers modeled
on  $\Gamma(E)$, where $E\to M$ is a fixed finite rank complex vector bundle over a closed, oriented manifold.  It is important to specify which sections are allowed.  From a Hilbert space point of view, it is easiest to work with $L^2$ sections, but of course such sections have no regularity.  In contrast, 
working with smooth sections forces us to deal with F\'echet spaces as fibers; since these spaces are tame 
in the sense of Hamilton, this is workable but more difficult.  As a reasonable compromise, we usually work
 with the Sobolev space $\calH = H^s(E)$ of $H^s$ sections for $s \gg 0$, as these sections are highly differentiable and form a Hilbert space.  
 
 We now have to decide on the structure group of $\calE$.  The first natural choice of $GL(\calH)$, the group of bounded automorphisms of $\calH$ with bounded inverse, is too large:  $GL(\calH)$ is contractible,
 so every $GL(\calH)$-bundle is trivial.  Fortunately, in the cases we consider, the transition functions lie in a
  gauge group or group of $\pdo$s which have nontrivial topology.  
 
To develop the analog of
finite dimensional Chern-Weil theory for, say, the gauge group Aut$(E)$,  we need (i)
an Ad-invariant analytic function $P$ on $\Endo(E) = {\rm Lie}(\Auto(E))$, and (ii) an $\Aut(E)$-connection $\nabla$ on $\calE.$  This data will give 
a characteristic class $[P(\Omega)]\in H_{dR}^*(\calM,\C)$, where $\Omega$ is the curvature of $\nabla.$  The same procedure works for a structure group of $\pdo$s.

The determination of all invariant polynomials or analytic functions on $\Endo(E)$ is an interesting, perhaps difficult, infinite dimensional version of classical invariant theory.  
To avoid this issue, we recall that the polynomials $A\mapsto \tr(A^k)$ generate the invariant polynomials on $\mathfrak u(n)$.  By the same arguments, any trace on $\Endo(E)$, i.e. a linear map $T:\Endo(E)\to \C$ with  $T[A,B] = 0$, will give characteristic classes $[T(\Omega^k)].$  The set of all traces is $HH^0(\End(E))$, the zeroth Hochschild cohomology group, which should be computable.  In any case, it is somewhat of a relief that the (nonlocal, not computable) operator trace is not a trace on $\Endo(E)$, since e.g. Id is not trace class.  Sidestepping again, we note that
$$A\mapsto \int_M \tr(A) \dvol$$
is a trace on $\Endo(E)$, where $\tr$ is the usual matrix trace and we have fixed a Riemannian metric on
 $M$.  Varying the metric presumably yields an infinite dimensional vector space of traces, but they are all of the same fundamental type.  Moreover, for $f\in  C^\infty(M)$, $A\mapsto \int_M \tr(A) f\dvol$ is a trace, and in fact for any distribution $D\in\calD'(M)$, $A\mapsto D(\tr(A))$ is a trace.  By the time we include distributions, locality is lost, so we will stick to the basic example:

\begin{defn}\label{defone} {\it The} $k^{\rm th}$ {\it component of the leading order Chern character of the $\Aut(E)$-bundle $\calE\to \calM$ is the de Rham class}
$$\cklo(\calE) = %\vol(S^*M)
\frac{1}{k!}\left[\int_M \tr(\Omega^k)\ \dvol\right]\in H^{2k}(\calM),$$
{\it where $\Omega$ is the curvature of an $\Aut(E)$-connection on $\calE$.
The leading order Chern character is $ch^{\rm lo}(\calE) = \sum_k\cklo(\calE)$.}
% and $\vol(S^*M)$ is the volume of the unit cosphere bundle of $M$.}
\end{defn}
We can similarly define leading order Chern classes.

\subsection{Leading order classes and mapping spaces}

It is well known that the tangent bundle $T\mapsnm$ is the pushdown of a finite rank bundle, as we now explain.  At a fixed $f\in\mapsnm$, take a curve $\eta(t)\in \mapsnm$ with $\eta(0) = f.$  For each $n\in N$,
$\dot\eta(t)(n)\in T_{f(n)}M$ gives the infinitesimal information in $\eta$ at $n$.  Thus an element of 
$T_f\mapsnm$ is a section $x\mapsto \dot\eta(t)(x)$ of $f^*TM\to N,$  so $T_f\mapsnm = \Gamma(f^*TN)
$, where we take all smooth sections  for the moment.  This is summarized in the diagram

\begin{equation}\label{diagram}\begin{CD} 
@.\textrm{ev}^*TM @>>> TM\\
@.@VVV   @VVV\\
@.\mapsnm\times N @>{\rm ev}>> M\\
@.@V\pi VV @.\\
T\mapsnm = \pi_*\ev^*TM@>>> \mapsnm  @.
\end{CD}\ \ \ \ \ \ \ \ \ \ \ \ \ \ \ \ \ \ \ \ \ \ \ \ \ 
\end{equation}

\smallskip
\noindent where $\ev:\mapsnm\times N\to M$ is the evaluation map $\ev(f,n) = f(n)$, $\pi$ 
is the projection,
 and $\pi_*$ is the pushdown functor in sheaf theory: $(\pi_*\ev^*TM)|_f = \Gamma(\ev^*TM|_{\pi^{-1} (f)}) = \Gamma(f^*TM).$

For $g$ in the connected component $T\mapsnm_f$ of $f$,  $g^*TM$ is noncanonically isomorphic to $f^*TM$, and so the sections of these bundles are noncanonically isomorphic.
  In a neighborhood of $f$, we can choose these isomorphisms smoothly.  This gives a local trivialization of $T\mapsnm$ and implies that on overlaps, the transition functions are given by 
 gauge transformations of the model fiber $\Gamma(f^*TM).$

Similarly, a bundle $E\to M$ induces a bundle $\calE = \pi_*\ev^*E\to \mapsnm.$  The Chern classes of $E$ are related to the leading order classes of $\calE$, as we now show; this has only been sketched before,
and the proof below is joint work with A.  Larra\'in-Hubach.

For $n\in N$, let $\ev_{n}:\mapsnm\to N$ be $\ev_{n}(f) = f(n)$.
%and let $\ev_f:N\to M$ be $\ev_f(n) = f(n).$  
As $n$ varies over $N$ and $f$ is fixed, $\ev_n^*E$ glues together to a bundle over $N$ which is precisely $f^*E $. % we denote this by $\ev_{\cdot,f}^*E = f^*E.$
Geometrically, for $\nabla$  a connection on $E$ with curvature $\Omega$,  $f^*E\to N$ has the connection $f^*\nabla$ with curvature $f^*\Omega$ which at $n$ equals $\ev_n^*\Omega$ at $f\in \mapsnm.$

\begin{prop}  Assume $N$ is connected and fix $n_0\in N.$  Let $E\to M$ be a complex vector bundle and set $\calE = \pi_*\ev^*E.$  Then
$$\cklo(\calE) = %\vol(S^*N)  
\vol(N) \ev_{n_o}^*c_k(E)\in H^{2k}(\mapsnm, \C).$$
\end{prop}

\begin{proof}
We have
\begin{eqnarray*} c_k^{\rm lo}(\calE)(f) &=& 
\int_{N} \tr (f^*(\Omega^k) )\ \dvol_N\\
&=& \int_{N} \tr(\ev_{n}^*(\Omega)^k) \ \dvol_N(n)\\
&=&  \int_N \ev^*_{n} \tr(\Omega^k)\  \dvol_N(n).
\end{eqnarray*}

 Let $\gamma_n:[0,1]\to N$ be a path from $n$ to $n_0$.  For $n$ not in the cut locus 
 $\calC_{n_0}$ of $n_0$, we can choose $\gamma_n$ so that 
  $\gamma_n(t)$ is smooth in $n$ and $t$.
Set
 $$F : (N\setminus \calC_{n_0}) \times  \mmaps(N,M) \times [0,1]\to M, \ F(n,f, t) = \ev_{\gamma_n(t)}(f) = f(\gamma_n(t)).$$
$F_n = F|_{(n,\cdot, \cdot)}$ is a homotopy from $\ev_n$ to $\ev_{n_0}$ depending
smoothly on $n$, so there is a chain homotopy $I_n$ with 
\begin{equation}\label{one} \ev^*_n\tr (\Omega^k) -\ev_{n_0}^*\tr (\Omega^k) = (d_{\mmaps}I_n + I_nd_{\mmaps\times [0,1]}) \tr F_n^*(\Omega^k)
\end{equation}
on $\Lambda^{2k}( \mmaps(N,M)).$
%(And using tr commutes with pullbacks like $\ev_n^*$.) 
Thus for $N' = N\setminus \calC_{n_0}$, 
%at a fixed $f\in \maps(N,M)$, 
\begin{eqnarray*} 
 \int_N \ev^*_n \tr(\Omega^k)\  \dvol_N(n) 
 &=&  \int_{N'} \ev^*_n \tr(\Omega^k)\  \dvol_N(n)\\
&=&  \int_{N'} \ev_{n_0}^*\tr[ (\Omega^k) + (d_\mmaps I_n + I_nd_{\mmaps\times [0,1]}) \tr F_n^*(\Omega^k)] \ \dvol_N(n),
\end{eqnarray*}
since the cut locus has measure zero in $N$. The last integrand, pulled back
to the interior of the cut locus in $T_{n_0}N$, extends continuously to the cut locus 
in $T_{n_0}N$, so the integral is well defined.

Using $d\tr(\alpha) = \tr(\nabla\alpha)$ for a Lie algebra valued form $\alpha$ and setting 
 $\tilde d = d_{{\rm Maps}\times [0,1]}$ gives
$$\tilde d\tr F_n^*(\Omega^k) = \tilde d F_n^* \tr(\Omega^k) = F_n^* 
d_{M}  \tr(\Omega^k)
= F_n^*\tr (\nabla(\Omega^k)) = 0,$$
by the Bianchi identity.  
Thus 
\begin{eqnarray*} 
\int_{N} \ev^*_n(\tr(\Omega^k))\ \dvol_N(n)  
%&=& \int_{S^*N'} \tr\sigma_0(\ev^*_n((\Omega^u)^k) d\xi\ dvol_N(n)\\
&=&  \int_{N'} \ev_{n_0}^*\tr (\Omega^k) \ \dvol_N(n) + 
d_\mmaps \int_{N'} I\tr F_n^*(\Omega^k) \ \dvol_N(n)\\
&=& \vol (N) \ev_{n_0}^*\tr (\Omega^k) 
+ 
d_\mmaps \int_{N'} I\tr F_n^*(\Omega^k)\ \dvol_N(n).
\end{eqnarray*}
Therefore
$$
\cklo(E\calG)=\left[ \int_N \ev^*_{n} \tr(\Omega^k)\  \dvol_N(n)\right]=
\vol(N) [ \ev_{n_0}^*\tr (\Omega^k) ]
= \vol(N) \ev_{n_0}^* c_k(E).
$$

\end{proof}

By \cite[Thm.~4.7]{lrst}, this lemma allows us to detect classes in $H^*(\mapsnm,\C).$

\begin{thm}  If $E\to M$ has $c_k(E)\neq 0$, then $\cklo(\calE)\neq 0$ in $H^*(\mapsnm, \C)$ for any $N$.
\end{thm}

The case $N = S^1$ is already important, as the cohomology of $LM$ is given by a cyclic complex construction based on Chen's iterated integral \cite{GJP}.  It would be interesting to relate these two approaches.

\subsection{Leading order classes and index theory}
${}$\\

{\it See the note at the beginning of the Introduction.}
\medskip

The $S^1$-Atiyah-Singer index theorem can be restated on loop space in a way that handles all isometric actions in one setting.  
This material extends work in  \cite{LMRT}.

  For a review, assume $M$ is a closed, oriented, Riemannian manifold which is spin and has an $S^1$ action via isometries. $S^1$ is also assumed to act on $(E,\nabla^E,h)\to M$ covering its action on $M$,
  where $\nabla$ is an equivariant connection which is hermitian for the hermitian inner product $h$.
 In this setup, the kernel
and cokernel of the twisted Dirac operator $\dir_{\nee}$ are representations of $S^1 = U(1)$. The $S^1$-index of
$\dir_{\nee}$ is % at $e^{i\theta}\in S^1$ is the trace of this character 
the corresponding element of the representation ring $R(S^1)$:
$${\rm ind}_{S^1}(\dir_{\nee}) = \sum(a_k^+ - a^-_k)t^k\in \Z[t,t^{-1}] = R(S^1),$$
where $t^k$ denotes the representation $e^{i\theta}\mapsto e^{ik\theta}$ of $S^1$ on $\C$, and 
$a_\pm^k$ are the multiplicities  of $t^k$ in the kernel and cokernel of $\dir_{\nee}.$
Equivalently, for each  $e^{i\theta}\in S^1$, we define
${\rm ind}_{S^1}( e^{i\theta},\dir_{\nee}) = \sum(a_k^+ - a^-_k)e^{ik\theta}.$

The Atiyah-Segal-Singer fixed point formula  computes
${\rm ind}_{S^1}(e^{i\theta}, \dir_{\nee}) $ in terms of data on the fixed point set of a particular $e^{-i\theta}$. (The minus sign is for convenience.)
As in \cite [Ch. 8]{BGV}, this can be rewritten as 
\begin{equation}\label{gind} {\rm ind}_{S^1}( e^{-i\theta}, \dir_{\nee})  =
 (2\pi i)^{-{\rm dim}(M)/2}\int_M 
\hat A_{\fraku}(\theta, \Omega_\fraku) {\rm ch}(\theta,\Omega_{\fraku}^E).
\end{equation}
Here $\Omega_\fraku$ is the equivariant curvature of the Levi-Civita connection on $M$, and $\hat A_{\fraku}(\theta, \Omega_\fraku) % = \hat A(\Omega_\fraku)(\theta)
\in\Lambda^*(M)$ is the equivariant $\hat A$-polynomial 
 $ \hat A(\Omega_\fraku)\in (\C[u]\otimes\Lambda^*M)^{S^1}$ evaluated 
at $\theta\in \fraku(1).$  Similarly, $\Omega_\fraku^E$ is the equivariant 
curvature of $\nabla^E$, and ch denotes the equivariant Chern character.
We condense the notation in  (\ref{gind}) to
\begin{equation}\label{ease}\overline{\rm ind}_{S^1}( \dir_{\nee})  =
 (2\pi i)^{-{\rm dim}(M)/2}\int_M^{S^1} 
\hat A_{\fraku}(\Omega_\fraku) {\rm ch}(\Omega_{\fraku}^E),
\end{equation}
with the bar on the left hand side indicating evaluation at $e^{-i\theta}$ and the right hand side evaluated at $\theta.$  For the trivial action, (\ref{ease}) is exactly the ordinary index theorem. 

As mentioned in the Introduction, the circle action $a:S^1\times M\to M$ induces $a^L:M\to LM$, 
$a^L(m)(\theta) = a(\theta, m).$  Denoting $a^L$ just by $a$, we get the class $a_*[M]\in H_n(LM,\Z)$ determined by the action.  $LM$ has the rotation action $r:S^1\times LM\to LM$, $r(\theta, \gamma)(\psi)
= \gamma(\theta+\psi).$  This action is via isometries for the $L^2$ metric on $LM$:
\begin{equation}\label{nat}
\langle X, Y\rangle_\gamma = \frac{1}{2\pi}\int_{S^1}\langle X(\theta), Y(\theta)\rangle_{\gamma(\theta)} d\theta,
\end{equation}
for $X, Y\in T_\gamma LM = \Gamma(\gamma^*TM\to S^1).$\footnote{Recall that we take $H^s$ sections, so this is a weak Riemannian metric on $LM$.}  The suitably averaged $L^2$ Levi-Civita connection is $r$-equivariant, so we can form the equivariant curvature $\wog$ and $\hat A(\wog)$ on $LM$.  The
$\hat A$-polynomial is an equivariant form on $LM$, denoted $\hat A(\wog)\in (\C[u]\otimes \Lambda^*
(LM))^{S^1}.$ Similarly, 
we can take the  equivariant curvature $\Omega^E_\fraku$ of $\nabla$, form the $L^2$ weak connection
  on $\calE = \pi_*\ev^*E\to LM$, average it to form the equivariant curvature $\woeg$, and then take its
  equivariant Chern character $\che(\woeg)$.
  
The map $a = a^L:M\to LM$ easily intertwines the actions $a$ on $E\to M$ and $r$ on $\calE\to LM$. From this, we easily get
$$a^*\hat A(\wog) = \hat A(\Omega_\fraku),\ \  a^*\che(\woeg) = \che(\Omega^E_\fraku).$$
From the $S^1$-index theorem, we therefore obtain:
\begin{thm}  \label{thm2} Let $M$ be a compact, oriented Riemannian spin manifold with an isometric $S^1$-action, and let $E$ be an equivariant hermitian bundle with connection $\nabla^E$ over $M$.  Then
$$\overline{\rm ind}_{S^1}(\dir_{\nee})  = (2\pi i)^{-{\rm dim}(M)/2} \int^{S^1}_{a_*[M]}\hat A(\wog)\che(\woeg).$$
\end{thm}  

It is easy to check that on the copy of $M$ sitting inside $LM$ as the constant loops, $\hat A(\wog),\che(\woeg)$ reduce to the $\hat A$-polynomial of $M$ and the Chern character of $E$, respectively.  Thus these forms give equivariantly closed extensions of these characteristic classes to $LM$.  A very different construction of an equivariant Chern character on $LM$ is given in \cite{twz}, based on ideas in \cite{B, GJP}.  It is natural to conjecture that the equivariant classes of the two Chern characters are the same; see \cite{tmc} for preliminary results.

The important point of Thm.~\ref{thm2} is that  the action information on the right hand side is now contained in the ``action class" $a_*[M]$, while the integrand only depends on the metrics on $M, E$.  In particular, the integrand is applicable to all isometric actions for fixed metrics.  

In fact,  we
can  remove this metric dependence from
the integrand as follows.  Let $\calB$ be
the space of metrics on $M$.  $\calB$ comes with a natural Riemannian $L^2$ metric $g^\calB$, given at $T_{g_0}\calB$ by
$$g^\calB (X,Y) = \int_M g_0^{ab}g_0^{cd}X_{ac}Y_{bd}
\dvol_{g_0}.$$
Here $X = X_{ac}dx^a\otimes dx^c\in T_{g_0}\calB$ and similarly for $Y$.
Thus $LM\times \calB$ has a metric $h$ which at $(\gamma, g_0)$ is the non-product metric determined by the 
 $L^2$ metric on $T_\gamma LM$ 
given by $g_0$ and by $g^\calB$ on $T_{g_0}\calB.$ 
 We extend the rotational action on $LM$ trivially to $LM\times \calB$, and so
obtain an equivariant curvature $\widetilde F_{\mathfrak u}
\in \Lambda^*(LM\times \calB, \End(TLM\oplus T\calB))$.  Let
$i_{g_0}:LM\to LM\times\calB$ be the inclusion $\gamma\mapsto (\gamma, g_0).$  By 
$i_{g_0}^*\widetilde F_\fraku$, we mean that we restrict 
$\widetilde F_\fraku$ to tangent vectors in $TLM$, and the ``endomorphism part" $A$ of $\widetilde F_\fraku$
is replaced by $P^{TLM}AP^{TLM}$, where 
$P^{TLM}$ is the 
$h$-orthogonal projection of $T(LM\times \calB)$ to $TLM.$  (Here we put the Sobolev topology on $LM$ and $\calB$ for some high Sobolev parameter, so that $TLM$ is closed in $T(LM\times \calB)$.)

\begin{thm}\label{newthm}  (i) 
$ i_{g_0}^*\widetilde F_\fraku  = \wog^{g_0}$, where $\wog^{g_0}$ is $\wog$ computed at the metric $g_0.$

(ii) 
 If $a$ is a $g_0$-invariant $S^1$ action on $M$,  then
$$\overline{\rm ind}_{S^1}(\dir) = (2\pi i)^{-\dim(M)/2}
\int^{S^1}_{i_{g_0,*}a_*[M]} \hat A(\widetilde F_\fraku ).$$
\end{thm} 

Since every $S^1$ action is isometric for some metric, this produces  a {\it universal equivariant
$\hat A$-polynomial} $\hat A(\widetilde\Omega_\fraku)$ on $LM\times \calB$, {\it i.e.}, an equivariantly closed form on $LM\times\calB$ such that the 
$S^1$-index is determined by (i) the universal $\hat A$-polynomial and (ii) the cycle of integration 
associated to the  action and compatible metric.
%We can bring in an equivariant  bundle $E\to M$ by taking $LM\times \calB\times \calC$, where $\calC$ is the space of hermitian metrics on $E$; the details are similar.

\begin{proof}  (i) The Levi-Civita connection on a Riemannian Hilbert manifold is characterized by
\begin{eqnarray}\label{5one}
2\ipo{\nabla_XY}{Z} &=& X\ipo{Y}{Z}+Y\ipo{X}{Z}-Z\ipo{X}{Y}\\
&&\qquad +\ipo{[X,Y]}{Z}+\ipo{[Z,X]}{Y}-\ipo{[Y,Z]}{X}.\nonumber
\end{eqnarray}
We 
use the right hand side of (\ref{5one}) with the $L^2$ inner product 
to motivate the definition of a connection $\widetilde \nabla$ on $LM\times\calB$, avoiding the issue of determining the topology for which the right hand side 
is a continuous linear functional of $Z$.

On $LM\times\calB$, we have at $(\gamma, g)$ 
$$(X,T)\ipo{(Y,S)}{(Z,R)} = X\ipo{Y}{Z} + T\ipo{S}{R} + \frac{1}{2\pi}\int_{S^1} \ipo{Y(\theta)}{Z(\theta)}_T
d\theta$$
using $\delta_T g = T.$  Here $\ipo{Y(\theta)}{Z(\theta)}_T = T_{ab}(\gamma(\theta)) Y^a(\theta)
Z^b(\theta).$
We also have
$\langle [(X,T),(Y,S)], (Z,R)\rangle = \langle [X,Y],Z\rangle + \langle [T,S],R\rangle.$ Therefore
we must have 
\begin{equation}\label{24one}
\langle \widetilde\nabla_{(X,T)}(Y,S), (Z,R)\rangle =
\langle \nabla^{LM}_XY,Z\rangle + \langle \nabla^\calB_TS,R\rangle +\frac{1}{2} \alpha^\sharp,
\end{equation}
where $\alpha^\sharp$ is the tangent vector $L^2$ dual to the one-form on $LM\times\calB$ given by
$$(Z,R) \mapsto \langle Y,Z\rangle_T + \langle X,Z\rangle_S - \langle X,Y\rangle_R,$$
and we use the shorthand $\langle Y,Z\rangle_T = \frac{1}{2\pi}\int_{S^1} \langle Y(\theta),Z(\theta)\rangle_T d\theta.$ 
Since $\int_\gamma \omega =
  \int_M \omega \wedge \pdg$ for one-forms $\omega$ (where for the Poincar\'e dual $\pdg$ we may assume that $\gamma$ is embedded\footnote{By a small perturbation of $\gamma$, at least if dim$(M)>2.$}), we have
\begin{eqnarray*}\langle X,Y\rangle_R = \frac{1}{2\pi}\int_{S^1} R_{ab}X^aY^bd\theta 
&=& \frac{1}{2\pi}\int_M R_{ab}X^aY^b d\theta\wedge \pdg \\
&=&\int_M g^{ac}g^{bd}R_{ab}X_cY_d *(d\theta\wedge \pdg) \dvol\\
&=&\langle R, *(d\theta\wedge \pdg)  X^\flat\hat\otimes Y^\flat\rangle,
\end{eqnarray*}
where $X^\flat\hat \otimes Y^\flat$ is the symmetric product of the one-forms $X^\flat=X_cdx^c, Y^\flat
=Y_d dx^d$ dual to $X, Y.$ We also have
\begin{eqnarray*}
\langle Y,Z\rangle_T &=& \frac{1}{2\pi}\ints T_{ab}Y^aZ^b =\frac{1}{2\pi}\ints g_{bd}g_{ac}T^{cd}Y^aZ^b\\
&=& \langle Z, g_{ac} Y^aT^{cd}\partial_d\rangle = \langle Z, Y^aT_a^{\ d}\partial_d\rangle\\
&=& \langle Z, i_YT\rangle,
\end{eqnarray*}
with $i_Y$  the interior product.
Similarly, $\langle X,Z\rangle_S = \langle Z, i_XS\rangle.$  We obtain that the connection
 $\wn$ on $LM\times\calB$ must be
\begin{eqnarray}\label{24two}
\lefteqn{ \wn_{(X,T)}(Y,S)_{(\gamma,g)} }\\
&=&  (\nlm_XY,0) +(0,\nabla^\calB_TS) + \frac{1}{2} \left[( (i_YT)^\sharp, 0) + ((i_XS)^\sharp, 0) 
- (0, *(d\theta\wedge \pdg)  X^\flat\hat\otimes Y^\flat)\right],\nonumber
\end{eqnarray}
where $\nabla^{LM}$ is computed for $g$.  Since the
last term on the right hand side of (\ref{24two}) is  
an endomorphism of $(Y,S)$, 
$\wn$ is a connection.  ($\wn$ is derived from (\ref{5one}), so it is a torsion free, metric connection for the $L^2$ metric.  However, (\ref{24two}) does not make sense for $(X,T), (Y,S)$ in
$L^2$, so we cannot claim that $\wn$ is the $L^2$ Levi-Civita connection.)
%%%%%%%%%%%%%%%%%

Denoting $(X,0)\in T(LM\times \calB)$ by $X$, the curvature  $\widetilde F$ of $\wn$ restricted to $LM\times \{g_0\}$ is
$i_{g_0}^*\widetilde F = P^{TLM}(\wn_X\wn_Y-\wn_Y\wn_X -\wn_{[X,Y]})P^{TLM}.$  We have
\begin{eqnarray*} P^{TLM}\wn_X\wn_Y Z&=& P^{TLM}\wn_X[(\nlm_YZ,0) - \frac{1}{2}(0, *(d\theta\wedge \pdg)  Y^\flat\hat\otimes Z^\flat)]\\
&=& P^{TLM} [(\nlm_X\nlm_YZ,0) -\frac{1}{2}(0,*(d\theta\wedge \pdg) X^\flat \hat \otimes
(\nlm_YZ)^\flat) + \frac{1}{2}([i_X(Y^\flat\hat\otimes Z^\flat)]^\sharp,0)]\\
&=& \nlm_X\nlm_YZ + \frac{1}{2}\langle X,Y\rangle Z.
\end{eqnarray*}
Since
$P^{TLM}\wn_{[X,Y]}Z = \nlm_{[X,Y]}Z$, we get
$i_{g_0}^*\widetilde F =
\widetilde \Omega.$

The equivariant curvature on $LM\times \calB$  is given by $\widetilde F_\fraku = \widetilde F +\mu$, with
 $\mu(\dg,0) = \calL_{(\dg,0)} - \widetilde\nabla_{(\dg,0)},$ since $(\dg,0)$ is the vector field of the action.  It 
 follows that $P^{TLM}\mu_{(\dg,0)} = \calL_{\dg} -\nlm_{\dg} = \mu_{\dg}.$  Thus 
$i_{g_0}^*\widetilde F_\fraku  = \wog.$

(ii)  By Thm \ref{thm2}, we have
$$ (2\pi i)^{{\rm dim}(M)/2} \overline{\rm ind}_{S^1}(\dir)  =\int^{S^1}_{a_*[M]}\hat A(\wog)
= \int^{S^1}_{a_*[M]} \hat A(i_{g_0}^*\widetilde F_\fraku)
= \int^{S^1}_{a_*[M]} i_{g_0}^*\hat A(\widetilde F_\fraku)
= \int^{S^1}_{i_{g_0,*}a_*[M]} \hat A(\widetilde F_\fraku).
$$

\end{proof}

%One can more or less explicitly identify $[\hat A(\wog)]\in H^n(LM,\R).$  Let $[z]$ be a real $n$-cycle on $LM$. For some $r\in\R$,  $r[z]$ has a representative submanifold $N$.  Then $ \frac{1}{r} \int_N^{S^1} \hat A(\wog) = \langle [\hat  A(\wog)], [z]\rangle$. 
%By the equivariant localization theorem \cite[Thm. 7.13]{BGV}, 
%$$\int_N^{S^1} \hat A(\wog)  = \int_{N\cap M}^{S^1} \frac{\hat A(\wog)}
%{\chi_\fraku(\nu_N^{LM})},
%$$
%where $\nu_N^{LM}$ is the normal bundle of the fixed point set ${\rm Fix}(N) = 
%[N]\cap M$ in $N$, and $\chi_\fraku$ is the equivariant Euler form.  The equivariant Chern character 

We now sketch the easier construction of a universal Chern character for a fixed bundle $E\to M.$  Let
$\calC = \{(\nabla^E, h^E)\}$, where $\nabla^E$ is a hermitian connection on $E$  for the hermitian metric $h^E.$
$\calC$ fibers over $\calH$, the space of hermitian inner products on $E$, with fiber $\calC_h$ modeled on 
$\Lambda^1(M, \End_h(E))$, with $\End_h(E)$ the space of $h$-skew-hermitian endormorphisms of $E$.
This fibration is locally trivial: (i) There is a coset $H$ of the $h$-unitary frame bundle of $E$ inside the bundle of $GL(n,\C)$-frames such that $A\in H$ iff $A$ takes
an $h$-orthonormal frame of $E$ to an $h'$-orthonormal frame; (ii) $C$ is $h'$-skew-hermitian iff $ACA^{-1}$ is $h$-skew-hermitian; (iii) For all $h'$ close to $h$, we can make a smooth choice of  $A = A_{h'}$ in a contractible neighborhood of the identity in $\Gamma(\End(E))$, giving smoothly varying
isomorphisms $\Lambda^1(M, \End_h(E))\simeq \Lambda^1(M, \End_{h'}(E)).$
 Thus $\calC$ is a Banach
or Fr\'echet manifold, a subset of $\calH\times \calA$, where $\calA$ is the space of connections on $E$.

The bundle $\calE = \pi_*\ev^*E\to LM$ pulls back to $p^*\calE\to LM\times \calC$ under the projection $
p:LM\times\calC\to LM.$  $p^*E$ has the connection given at $(\gamma, \nabla, h)$ by
$$\widehat \nabla_{(X, \omega,T)}% (X_2, T_2, \omega_2) = \wnlm^{(\nabla, h$$
s = \wnen_X s + \delta_{\omega}s +  \delta_{T}s,$$
where $\wnen$ is the connection on $\calE$ associated to $\nabla$, and $\delta_\omega, \delta_T$ denote trivial connections in the $\calA, \calH$ directions.  
For $\hat X_i = (X_i,0,0)$, it follows immediately that $\widehat \Omega(\hat X_1, \hat X_2)_{(\gamma, \nabla, h)}
= \widetilde \Omega^\nabla(X_1, X_2)_{\gamma}$ in the obvious notation. For $\widetilde\Omega^\nabla = (\widetilde
\Omega^\nabla)_{ij\ b}^{\ \ a} dx^i\wedge dx^j \otimes e^a\otimes e_b,$ for $\{e_b\} $ a local frame of $E$ with dual frame $\{e^a\}$, we define
\begin{eqnarray*}
\Tr(\widehat\Omega)_{(\gamma,\nabla, h)} &=& \left(\frac{1}{2\pi}\int_{S^1} h_{ac}h^{cb} 
[(\widehat\Omega^\nabla)_{ij\ b}^{\ \ a}]_{ (\gamma,\nabla, h)}(\theta) \ d\theta \right)dx^i\wedge dx^j\\
&=&
\left(\frac{1}{2\pi}\int_{S^1} h_{ac}h^{cb} 
(\widetilde\Omega^\nabla)_{ij\ b}^{\ \ a} (\gamma)(\theta) \ d\theta \right)dx^i\wedge dx^j\\
&=& 
 \left(\frac{1}{2\pi}\int_{S^1} h_{ac}h^{cb} 
(\Omega^\nabla)_{ij\ b}^{\ \ a} (\gamma(\theta))\ d\theta\right) dx^i\wedge dx^j .
\end{eqnarray*}
The same result holds for $\widehat \Omega_\fraku.$
$\Tr(\widehat\Omega_\fraku^k)$ and $\che(\widehat\Omega_\fraku)$ are defined similarly.  Let $(\nabla, h)$ be an equivariant connection and hermitian metric for an action of $S^1$ on $E$.  For the inclusion $j: LM\to LM\times \calC, \gamma
\mapsto (\gamma, \nabla, h)$, we have $j^*\che(\widehat \Omega_\fraku) = \che(\widetilde\Omega^\nabla_\fraku),$ where $\widetilde\Omega^\nabla_\fraku $
is the equivariant curvature of $\calE$ associated to $\nabla.$  Thus $\che(\widehat\Omega_\fraku)$ is a universal equivariant Chern character for $E\to M.$ 

Finally, one can combine the equivariantly closed forms $\che(\widehat\Omega_\fraku)\in
(\C[u]\otimes\Lambda^*(M\times\calC))^{S^1}$ with the universal $\hat A$-form in 
$(\C[u]\otimes \Lambda^*(LM\times \calB))^{S^1}$ by 
pulling them back to 
$(\C[u]\otimes \Lambda^*(LM\times\calB\times \calC))^{S^1}$.  In summary:
 
 \begin{thm} \label{bigtwo}(i) Let $(M,g)$ have an isometric $S^1$ action  and let $(E,\nabla,h) \to M$ be an equivariant
 bundle with an $h$-hermitian equivariant connection $\nabla.$  
Let $j = j_{(g, \nabla, h)} :LM\to LM\times\calB\times \calC$ be the injection $j(\gamma) = (\gamma, g, \nabla, h).$ The equivariantly closed form
 $$\hat A(\widetilde F_\fraku) \che(\widehat\Omega_\fraku) \in 
 (\C[u]\otimes \Lambda^*(LM\times\calB\times \calC))^{S^1}$$ 
 has
 $j^* [\hat A(\widetilde F_\fraku) \che(\widehat\Omega_\fraku)] = \hat A(\widetilde\Omega^g_\fraku)
 \che(\widetilde\Omega^\nabla_\fraku).$

 (ii) We have
 $$\overline{\rm ind}_{S^1}(\dir_{\nee})  = (2\pi i)^{-{\rm dim}(M)/2} \int^{S^1}_{j_*a_*[M]}
 \hat A(\widetilde F_\fraku) \che(\widehat\Omega_\fraku).$$
 \end{thm}

Thus the ``universal index form" $\hat A(\widetilde F_\fraku) \che(\widehat\Omega_\fraku)$ determines the 
$S^1$-index for every action of the circle on $E\to M.$

\subsection{Flat fibrations and Gromov-Witten invariants}

The fibration $\pi$ in (\ref{diagram}) is trivial.  
In this subsection, we discuss to what extent leading order classes appear in nontrivial fibrations and give applications to Gromov-Witten theory.

A finite rank bundle $E\to M$ on the total space of a fibration $Z\to M\stackrel{\pi}{\to} B$ of manifolds gives rise to an infinite rank bundle $\calE
= \pi_*E\to B$ with $\calE_b = \Gamma(E|_{\pi^{-1}(b)}).$  Fix a connection $D$ for the fibration, {\it i.e.}, a complement to 
the kernel of $\pi_*$ in $TM.$  The connection $\nabla$ pushes down to a connection $\pi_*\nabla = \nabla'$ on $\calE$ by
\begin{equation}\label{pdconn}\pi_*\nabla_X(s')(m) = \nabla_{X^h}(\tilde s)(m),
\end{equation}
where $X^h$ is the $D$-horizontal lift of $X\in T_bB$ to $T_mM$,
$s'\in \Gamma(\calE)$,  and $\tilde s\in \Gamma(E)$ is defined by 
$\tilde s(m) = s'(\pi(m))(m).$  The curvature $\Omega'$ of $\nabla'$ satisfies
\begin{eqnarray*} \Omega'(X,Y) &=& \pi_*\nabla_X\pi_*\nabla_Y - \pi_*\nabla_Y\pi_*\nabla_X - \pi_*\nabla_{[X,Y]}\\
&=& \nabla_{X^h}\nabla_{Y^h} - \nabla_{Y^h}\nabla_{X^h} -\nabla_{[X,Y]^h}\\
&=& \nabla_{X^h}\nabla_{Y^h} - \nabla_{Y^h}\nabla_{X^h} - \nabla_{[X^h, Y^h]} 
+ \left(\nabla_{[X^h, Y^h]} - \nabla_{[X,Y]^h}\right)\\
&=& \Omega(X^h, Y^h) + \left(\nabla_{[X^h, Y^h]} - \nabla_{[X,Y]^h}\right).
\end{eqnarray*} 
$\Omega(X^h, Y^h)$ is a zeroth order or multiplication operator, so in
 general, $ \nabla_{[X^h, Y^h]} - \nabla_{[X,Y]^h}$ and hence $\Omega'$
acts on  the fibers of $\calE_b$ as a first order differential operator.

The leading order trace is only a trace on differential operators (or $\pdo$s) of nonpositive order.  Thus we are naturally led to restrict attention to fibrations with flat or integrable connections, which by definition means 
$[X^h, Y^h]- [X,Y]^h = 0.$  Flat fibrations appear in Gromov-Witten theory and for mapping spaces, but
the setup for the families index theorem involves non-flat fibrations; it is a major drawback that 
 our approach does not apply to this case.

  We summarize the setup for  Gromov-Witten theory, with more details in 
\cite{LMRT}. Let $M$ be a closed symplectic manifold with a generic compatible almost complex structure.
For $A\in H_2(M,\Z)$, set
$C^\infty_0(A) = \{f:\BbP^1\to M| f\in C^\infty, f \ {\rm simple}, f_*[\BbP^1] = A\}.$  Set  $\BbP_k^1 = \{(x_1,\ldots, x_k)\in (\BbP^1)^k: x_i\neq x_j\ {\rm for}\ i\neq j\}.$
For
fixed $k \in \Z_{\geq 0}$, set
 $$C^\infty_{0,k}(A) = (C^\infty_0(A) \times \BbP_k^1)/\Aut(\BbP^1).$$
 $C^\infty_{0,k}(A)$ is an infinite dimensional manifold of either Banach or Fr\'echet type.
Denoting an element of $C^\infty_{0,k}(A)$ by $[f, x_1,\ldots, x_k]$, we set the moduli space 
of pseudoholomorphic maps to be 
$\calM_{0,k}(A) = \{[f, x_1,\ldots, x_k]: f\ {\rm is\ pseudoholomorphic}\}.$  $\calM_{0,k}(A)$ is
a smooth, finite dimensional, noncompact manifold.

The forgetful map $\pi = \pi_k:C^\infty_{0,k}(A)\to C^\infty_{0,k-1}(A)$  given by
$[f, x_1,\ldots, x_{k-1}, x_k]\mapsto [f, x_1,\ldots, x_{k-1}]$   is a locally trivial smooth 
fibration. It is shown in \cite{LMRT} that $\pi$ is flat.  As a result, we can relate Gromov-Witten invariants on 
$\mok$ to leading order classes on $\mokm$, at least in the case where the boundary of these moduli spaces is homologically small, i.e., the boundaries of the compactified moduli spaces have big enough codimension.  This occurs for $M$ semipositive, e.g.~for many smooth projective Fano varieties.  In this case, the Gromov-Witten invariants
$\langle \alpha_1^{\ell_1}\ldots\alpha_k^{\ell_k}\rangle$, for $\alpha_i\in H^*(M,\C)$, are given by the expected integral 
$\int_{\mok} \ev^*(\alpha_1^{\ell_1}\wedge\ldots \wedge\alpha_k^{\ell_k})$,
where $\ev[f,x_1,\ldots, x_k] = (f(x_1),\ldots,f(x_k)).$ 
  This is a very special case, as usually GW invariants involve the virtual fundamental class of the compactified moduli space.  

To state a result, let $\alpha_i$ be elements of the  even cohomology of $M$.  Since the Chern character $ch:K(M)\otimes \C\to H^{\rm ev}(M,\C)$ is an isomorphism, $\alpha_i
= ch(E_i)$ 
for a virtual bundle $E_i$.
Pullbacks and pushdowns of the $E_i$ are well defined virtual bundles.  Let $\pi_*ch(E_k)$ be the usual pushdown/integration over the fiber of $ch(E_k).$  In \cite{m-v}, this class is called the string  Chern class $ch^{\rm str}(\calE_k)$ of $\calE_k = \pi_*\ev^*E_k.$  Recall that the leading order Chern character is given in Def.~2.1. Finally, set $E_i^\ell = E_i^{\otimes \ell}.$  

%For notational convenience, we will assume that $E_i$ is an actual bundle and $\alpha_i = c_{r_i}(E_i)$ for 
%some integers $r_i,$  as it is easy to extend the calculation below to the general case.  

\begin{thm} \label{pp} Let $\alpha_i\in H^{\rm ev}(M,\C)$ satisfy $\alpha_i = ch(E_i)$
for $E_i\in K(M)$.  Set
$\calE_i = \pi_* \ev_i^*E_i\to\mokm.$  Then
\begin{eqnarray*}
\langle\alpha_1^{\ell_1}\ldots\alpha_{k}^{\ell_{k}}\rangle_{0,k}  
&=&  \langle ch^{\rm lo}(\calE_1^{\ell_1})\cdots
ch^{\rm lo}(\calE_{k-1}^{\ell_{k-1}})ch^{\rm str}(\calE_k)\rangle_{0,k-1}.
\end{eqnarray*}
\end{thm}

GW invariants have been used very successfully to distinguish symplectic structures on manifolds.  The leading order classes exist on the larger space $C^\infty_{0,k}(A)$.  There may be other symplectically defined 
cycles in this space that could be used similarly.  For example, the moduli spaces are minima for the holomorphic energy functional on $C^\infty_{0,k}(A)$; perhaps moduli spaces of nonminimal critical maps 
contain new homological information detected by leading order classes.

\subsection{Applications to loop groups and Donaldson invariants}

We briefly sketch other applications of leading order classes from \cite{LMRT}.

Loop groups $\Omega G$ are of course a very special mapping space.  The generators of $H^*(\Omega G, \R)$ for $G$ compact are known \cite[\S4.11]{Seg}.  As stated below, these generators are equal to certain 
leading order 
Chern-Simons  classes or equivalently Chern-Simons string classes, which are defined for a pair of connections just as in finite dimensions.  We start with a degree $k$ ${\rm Ad}_G$-invariant polynomial on the Lie algebra $\frakg$ of $G$.  For $G = U(n),$ 
$f$ is in the algebra generated by the polarization of $A\mapsto \Tr(A^k)$.  Just as with leading order Chern classes, we can associate a leading order class $f^{\rm lo}$ to any pushdown bundle $\calE\to B$, where 
$E\to M$ is a $G$-bundle and $M\to B$ is a Riemannian fibration.  While all this works for principal bundles, to fit with the previous setting of vector bundles, we choose a faithful Lie algebra representation on a finite dimensional vector space $V$, let
 $h:G\to\Aut(V)$ be the exponentiated representation, and work on the associated vector bundle $E
 \times_h V\to M.$  In particular, in Def.Ä \ref{defone}, we just replace $\tr(\Omega^k)$ with
 $f(\Omega,\ldots,\Omega).$

Given a pair 
of connections $\nabla_0, \nabla_1$ on $E\to M$ with connection one-forms $\omega_0, \omega_1$ and a 
Riemannain fibration $Z\to M\to B$ with fiber $Z_b$ over $b$, we define
$$CS^{\rm lo}_f (\pi_*\nabla_0,\pi_*\nabla_1) = \int_0^1 \int_{Z_b} f((\omega_1-\omega_0), \Omega_t,\ldots
\Omega_t)\dvol_{Z_b}\in \Lambda^{2k-1}(B),$$
with $k-1$ occurrences of $\Omega_t$, where $\Omega_t = d\omega_t+ \omega_t\wedge\omega_t, \omega_t = t\omega_0+ (1-t)\omega_t.$  As 
usual, \\
$d_BCS^{\rm lo}_f (\pi_*\nabla_0,\pi_*\nabla_1)
  = f^{\rm lo}(\pi_*\Omega_0) - f^{\rm lo}(\pi_*\Omega_1)$, so 
the leading order Chern-Simons forms are closed provided the leading order Chern forms for $\nabla_0, \nabla_1$ vanish pointwise.

To build leading order CS classes on $\Omega G$, we use the fibration (\ref{diagram}) with $N=S^1, M=G.$
Let $G$ have Lie algebra $\mathfrak{g}$ and Maurer-Cartan form
$\theta^G$. Choose 
 $h:G\to\Aut(V)$ as above.
 For $\underline{V}\to G$ the trivial vector bundle $G\times V\to G$, we can view 
$h$ as a gauge transformation of $\underline{V}.$
Let $\nabla_0=d$ be the trivial connection on $\underline V$, and let $\nabla_1= h\cdot \nabla_0 =
h^{-1}d h$ be the gauge transformed connection.   Since the connections are flat, the CS classes
$CS^{\rm lo}_f (\pi_*\nabla_0,\pi_*\nabla_1)  \in H^{2k-1}(\Omega G,\C)$ are defined. Similarly, CS string classes are given by integrating over the fiber $S^1$, so
$CS^{\rm str,\calE}_f(\ev^*\nabla_0,\ev^*\nabla_1) = \pi_*CS_f^{\ev^*E}(\ev^*\nabla_0,\ev^*\nabla_1) \in H^{2k-2}(\Omega G,\C)$ is defined for $E\to G.$

To state the results, let $\chi$ be the vector field on $\Omega G$ associated to the rotation action on loops:
$\chi(\gamma) (\theta) = \dot\gamma(\theta).$  Let $i_\chi$ denote the interior product.

\begin{thm} Let $\mathcal V = \pi_*\ev^*\underline V.$ Then $H^*(\Omega G, \R)$ is generated by 
$$CS_{f_i}^{{\rm str}, \mathcal V}
( \ev^*\nabla_0, \ev^*\nabla_1) = i_\chi CS_f^{{\rm lo},\mathcal V}( \pi_*\ev^*\nabla_0, \pi_*\ev^*\nabla_1).$$
\end{thm}
\bigskip

To describe the relationship between leading order classes and Donaldson invariants, we review the basic setup.  Let $P\to M$ be a principal $G$-bundle over a closed manifold $M$ for a
 compact semisimple group $G$. We denote by 
$\calas$, resp. $\calG$,  the space of irreducible connections on $P$, resp. 
 the gauge group of $P$. For a connection $A$ on $P$, let $d_A:
{\rm Lie}(\calG) = \Lambda^0(M, \adP)\to \Lambda^1(M, \adP)$ be the covariant derivative associated to $A$ on the adjoint bundle $\adP = P\times_{\rm Ad}\frakg.$  Then the vertical 
space of  $\calas\to \calB^* = \calas/\calG$ at $A$ is Im$(d_A)$. The orthogonal complement 
$\keras$ forms the horizontal space of a connection $\omega$
on $\calas\to\calB^*.$ 
Let $\Omega$ be the curvature of this connection.  Let $G_A = (d_A^*d_A)^{-1}$ be the Green's operator associated to $d_A.$

\begin{lem} \label{appeared} For $X, Y$ horizontal tangent vectors at $A$, we have 
$$\Omega_A(X,Y) = -2 G_A *[X,*Y]\in {\rm Lie}(\calG) = \Lambda^0(M, {\rm Ad}\  P).$$
\end{lem}

${\rm Lie}(\calG)$ can be thought of as an algebra of multiplication
operators
 via the injective adjoint  representation of
$\frakg$. Equivalently, we can pass to the $\calG$-vector bundle
${\rm Ad}\ \calas = \calas\times_{\rm Ad} {\rm Lie}(\calG)$ with fiber ${\rm Lie}(\calG)$ and 
take the leading order classes $c_k^{{\rm lo}, ({\rm Ad}  \calas)\otimes \C}$ of its associated connection $d{\rm Ad}(\omega)$, whose curvature   $[\Omega, \cdot]$ 
is usually  denoted just by $\Omega.$
  Either way, the leading order Chern form $\cklo(\Omega)$ of $\calas\to\calB^*$ is
given by
$\int_M \tr(\Omega^{\wedge k})\dvol$ for some Riemannian metric on $M$.   Here $C^{\wedge k}$ is the endormophism on $\Lambda^k V$ determined by an endormorphism $C$ on $V$.  Below, we denote $\Omega^{\wedge k}$ by $\Omega^k$, with the caution that this is not the same as the $\Omega^k$ occurring in the
Chern character. 

On 4-manifolds, Donaldson invariants are built from his $\nu$ and $\mu$ classes in $H^*(\calM,\Z)$, where $\calM$ is the moduli space of self-dual connections.  In fact, these classes are constructed on $\calB^*$ and then 
restricted to $\calM.$  By comparing with explicit calculations in \cite{dk}, we get

\begin{prop}As differential forms, $\nu$ equals $p_1^{{\rm lo},\calas}(\Omega) = - 
c_2^{{\rm lo},({\rm Ad} \calA^*)
\otimes \C}(\Omega)$ up to a constant. 
\end{prop}

%As mentioned before Def. \ref{defone}, we can define leading order classes as currents on $H^*(B,\C)$ for a fibration $M\to B.$  

For the $\mu$ classes, we take 
$a\in H_2(M,\Q)$, and  Donaldson's map
$\mu:H_*(M,\Q)\to H^{4-*}(\calM,\Q)$. 
 Recall that $\mu(a) = i^*(\nu/a)$, for the slant product $\nu/:
H_*(M,\Q)\to H^{4-*}(\calB^*,\Q)$ and $i:\calM\to\calB^*$ the inclusion.  In particular,
$\nu = \mu(1)$ for
$1\in H_0(M).$  
By  \cite[Prop. 5.2.18]{dk}, the two-form $C_\omega\in \Lambda^2(\calM)$ representing $\nu/a = \nu/\PD^{-1}(\omega)$ and 
hence
$\mu(a) $ is given at
$[A]\in \calM$ by
\begin{equation}\label{mu}C_\omega(X,Y) = \frac{1}{8\pi^2}\int_M \tr(X\wedge Y)\wedge \omega
+ \frac{1}{2\pi^2}\int_M \tr(\Omega_A(X,Y) F_A)\wedge \omega,
\end{equation} 
where $F_A$ is the curvature of $A$.  On the right hand side, we use any $A\in [A]$ and 
$X, Y\in T_A\calA^*$ with  $d_A^*X = d_A^*Y = 0.$

As mentioned before Def \ref{defone}, 
there is a leading order class associated to any distribution or zero current $\Lambda$ on $C^\infty(M)$, given pointwise by
$c_k^{{\rm lo}, \Lambda} = \Lambda (\tr(\Omega^k)),$
where $\Omega$ is the curvature of a connection taking values in the Lie algebra of 
a gauge group, as in this section.
In particular, for a fixed $f\in C^\infty(M)$ we have the characteristic class
$\int_M f\cdot \tr(\Omega^k).$
   We can just as well consider $\tr(\Omega^k)$ as a zero-current acting on $f$. Looking back at (\ref{mu}), we can consider the two-currents  
\begin{equation}\label{curr}
\tr(X\wedge Y), %= \tr(\Omega^0_A(X,Y)X\wedge Y), 
\ \tr(\Omega_A(X,Y) F_A),
\end{equation}
for fixed $X, Y$.  Thus we can consider $C$ as an element of 
$\Lambda^2(\calM, \calD^2)$, the space of two-current valued two-forms on $\calM.$

Because these two-currents are Ad${}_{\calG}$-invariant, the usual Chern-Weil proof shows that $C(\omega) = C_\omega$ is closed.  $C$ is built from Ad${}_G$-invariant functions,
 but only the first term in (\ref{curr}) comes from an invariant polynomial in 
${\rm Lie}(\calG)^\calG.$ Nevertheless, we 
interpret  (\ref{mu}) as a sum of
``leading order currents" evaluated on $\omega.$  

\begin{prop} For $a\in H_2(M^4,\Q)$,
a representative two-form for  Donaldson's $\mu$-invariant $\mu(a)$ is given by
evaluating the leading order two-current 
$$ \frac{1}{8\pi^2}\int_M \tr(X\wedge Y)\wedge \cdot
+ \frac{1}{2\pi^2}\int_M \tr(\Omega_A(X,Y) F_A)\wedge \cdot $$
%$$\tr(X\wedge Y) = \tr(\Omega^0_A(X,Y)X\wedge Y), \ \tr(\Omega_A(X,Y) F_A)$$
on any two-form Poincar\'e dual to a.
\end{prop}

As with Gromov-Witten theory, there may be other significant cycles in $\calB^*$ not in $\calM$ that could be detected by these leading order classes/currents.

\begin{rem}  Except for the gauge theory case,  the infinite rank bundles in this section have all been
 pushdowns of finite rank bundles on the total space of a trivial fibration as in (\ref{diagram}).  In contrast, the Families Index Theorem setup involves a superbundle with superconnection $(E,\nabla) \to M$ on the total space of a nontrivial fibration of manifolds
 $Z\to M\stackrel{\pi}{\to} B$ with $Z$ spin.   It would be very interesting to recast this theorem in terms of the infinite rank pushdown bundle $\pi_*E.$  $\pi_*\nabla$ has connection and curvature forms taking values in ${\rm End}
 (E|_{\pi^{-1}(b)})$ for $b\in B$, which is not very exciting.  However, the index bundle IND($\dir_{\nee}$) 
 of the family of 
 twisted Dirac operators on the fibers  is a subbundle of $\pi_*E$; with respect to the splitting of $\pi_*E$ into the index bundle and its orthogonal complement, the curvature and connection forms of $\pi_*\nabla$ decompose into matrices with entries
locally taking values in zeroth order $\pdo$s in a local trivialization.  Thus the group $\pdoos$ of zeroth order invertible $\pdo$s is related to the Families Index Theorem.  

Because the fibration $Z\to M\to B$ is nontrivial  in general, its structure group is the diffeomorphism group of $Z$, so it appears that the  $\pdo$s glue via Fourier integral operators (FIOs).  However, although the fibers $E|_{\pi^{-1}(b)}$ in different trivializations are related by
bundle maps covering diffeomorphisms of $Z$, it seems that the sections of the trivialized fibers are canonically isomorphic.  Thus a $\pdo$ on a fiber in one trivialization is equivalent to a $\pdo$ on the fiber in a different trivialization.  

As a result, it is not clear at present if  one has to extend the structure group from the group $\pdoos$  to some larger group of FIOs.
  For a discussion of invertible FIOs and their Lie algebra, see \cite{MM}, and for a discussion of perhaps a smaller Lie group, see \cite[\S 6]{Rosenberg}.
\end{rem}

 \section{Wodzicki classes and applications}

In this section we discuss characteristic classes on infinite rank bundles built from the Wodzicki residue, the only trace on the full algebra of $\pdo$s acting on sections of a fixed bundle.  We will see that the Pontrjagin or Chern classes of these bundles always vanish, but the associated Wodzicki-Chern-Simons classes can be nonzero.  We will then use these WCS classes to study diffeomorphism groups of a class of 5-manifolds.

In particular, we will find several classes of $5$-manifolds $\bmk$ with $\pi_1(\diff(\bmk))$ infinite.  In general, there seems to be little in the literature about the homotopy type of $\diff(M)$ once dim$(M) \geq 3$.

\subsection{Wodzicki-Chern-Simons classes}
As motivation, we have noted that $T\mapsnm$ is a gauge group bundle, {\it i.e.}, on the component of a fixed 
$f\in\mapsnm$, the transition functions lie in the gauge group $\calG$ of $f^*TM\to N.$  Thus any 
$\calG$-connection will have connection one-form and curvature two-form taking values in 
${\rm Lie}(\calG),$ an algebra of bundle endomorphisms/multiplication operators. However, the Levi-Civita connections of the natural Riemannian geometry of $\mapsnm$ have connection  and curvature forms taking values in a larger group of $\pdo$s.  This is similar in spirit to a finite rank hermitian bundle with a non-unitary connection.  In the finite rank case, the structure group $GL(n,\C)$ deformation retracts onto $U(n)$, so any connection can be unitarized.  In our case, the 
relevant group $\pdoos$ of invertible zeroth order $\pdo$s acting on sections of e.g.~ $E=f^*TM$ does not retract onto the gauge group. 

$\pdoos$ seems to be an important group in infinite dimensional geometry.  It is the intersection of the algebra of all $\pdo$s with the group 
$GL(\Gamma(E))$ and so is the largest group of $\pdo$s consisting of bounded invertible operators with 
bounded inverses. 
The Lie algebra of $\pdoos$ is $\pdoz$, the algebra of $\pdo$s of nonpositive order; see  \cite{paycha}, where we first learned of the importance of this group and its Lie algebra. Thus we are forced to deal with these $\pdo$-connections directly, and the Wodzicki residue is worth incorporating into Chern-Simons theory.   

Recall that a classical $\pdo$ $P$ acting on sections of $E\to M^n$ has an order $\alpha\in \R$ and a symbol expansion $\sigma^P(x,\xi) 
\sim \sum_{k=0}^\infty \sigma^P_{\alpha-k}(x,\xi),$ where $x\in M, \xi\in T^*_xM$, and $\sigma^P_{\alpha-k}(x,\xi)$ is homogeneous of degree $\alpha-k$ in $\xi.$  For  $(x,\xi)$ fixed , $\sigma^P(x,\xi),
 \sigma_{\alpha-k}^P(x,\xi) \in {\rm End}(E_x).$ The Wodzicki residue of $P$ is
\begin{equation}\label{resw}
\resw(P) = \frac{1}{(2\pi)^n} \int_{S^*M} \tr \sigma_{-n}^P(x,\xi) d\xi\ dx,
\end{equation}
where $S^*M$ is the unit cosphere bundle over $M$ with respect to a fixed Riemannian metric. It is nontrivial that $\resw$ is independent of coordinates and defines a trace: $\resw[P,Q] = 0.$  The
$\sigma^P_{\alpha-k}$ are computable microlocally at each $(x,\xi)$, which is crucial for us.  In contrast, the equivalent definition $\resw(P) = {\rm res}_{s=0}\Tr(\Delta^{-s}P)$, for any positive order, positive elliptic operator $\Delta$ on $\Gamma(E)$, shows that the Wodzicki residue is a regularized trace, and makes the local expression (\ref{resw}) all the more remarkable.  

Since the computation  complexity of $\sigma_{-n}$ grows exponentially with $n$, we will just consider loop spaces ($N=S^1$).  As a trace, the Wodzicki residue is an Ad-invariant polynomial on $\pdoos$, so we can define Wodzicki-Chern or residue classes for any $\pdoos$-connection on $LM$ by
\begin{equation}\label{wc}
c_k^W(TLM) =\frac{1}{k!}
\left[ \int_{S^*S^1}\tr\sigma_{-1}(\Omega^{k}) \ d\xi  dx\right]\in H^{2k}(LM,\C).
\end{equation}

These classes always vanish. For $c_k^W(TLM)$ is independent of the connection, and as a gauge bundle,
$TLM$ admits a gauge connection whose curvature $\Omega$ takes values in multiplication operators, an especially simple subset of $\pdoz.$  The symbol of a multiplication operator is just the operator itself, so
$\sigma_{-1}(\Omega^k) = 0.$  (It is conjectured that the residue classes vanish for all $\pdoos$-bundles.)

Thus we are forced to consider Wodzicki-Chern-Simons (WCS) forms:
\begin{defn}\label{wcsdef}
The k${}^{\rm th}$ {\it Wodzicki-Chern-Simons (WCS) form} of two $\pdo_0^*$-connections 
$\nabla_0,\nabla_1$ on $TLM$ is
\begin{eqnarray}\label{5.22}
CS^W_{2k-1}(\nabla_1,\nabla_0) &=&\frac{1}{k!}
 \int_0^1 \int_{S^*S^1}\tr\sigma_{-1}((\omega_1-\omega_0)\wedge 
(\Omega_t)^{k-1})\ dt\\ 
&=&\frac{1}{k!} \int_0^1 {\rm res}^{\rm w} 
[(\omega_1-\omega_0)\wedge 
(\Omega_t)^{k-1}]\ dt.\nonumber
\end{eqnarray}
\end{defn}
As usual, $d CS^W_{2k-1}(\nabla_1,\nabla_0) = c_k^W(\nabla_0) - c_k^W(\nabla_1).$  Therefore,
if $c_k^W(\nabla_0) = c_k^W(\nabla_1) = 0$ pointwise, we get WCS classes
$CS^W_k(TLM)\in H^{2k-1}(LM,\C).$  Of course, $TLM$ is a real bundle, but unlike in finite dimensions, there is no {\it a priori} reason for the WCS classes to vanish if $k$ is odd.

Finally, one might wonder if there are traces on $\pdoos$ besides the leading order trace (and its distributional variants) and the residue trace.  In fact, it is a theorem of \cite{lesch-neira,paycha-lescure} that there are no more traces.  However, analogous to the Pfaffian for ${\mathfrak s}{\mathfrak o}(n)$, 
there could certainly be Ad-invariant polynomials not built from traces
on $\pdoz$,  or on the
full algebra of $\pdo$s, or on a geometrically interesting subalgebra.  One step in this direction 
is the residue determinant in \cite{scott}, but a complete theory is unknown at present. 

\subsection{Levi-Civita connections on $LM$}
If $M$ has a Riemannian metric $g$, $LM$ has the $L^2$ metric
(\ref{nat}), which was important for the $S^1$-index theorem discussion in \S2.  On its own, this metric is not so interesting: its curvature  $\Omega(X,Y)_\gamma(\theta) = \Omega^M(X(\theta),Y(\theta))_
{\gamma(\theta)}$ contains no more information than the curvature of $M$.  It is much more fruitful to pick a 
parameter $s\gg 0$ and define the $s$-Sobolev or $H^s$-metric by
\begin{equation}\label{eq:Sob1}
\langle X,Y\rangle_{s}=\frac{1}{2\pi}\int_0^{2\pi} \langle(1+\Delta)^{s}
X(\alpha),Y(\alpha)
\rangle_{\gamma (\alpha)}d\alpha,\  X,Y\in \Gamma(\gamma^*TM).
\end{equation}
Here $\Delta=D^*D$, with $D=D/d\gamma$ the covariant derivative along
$\gamma$.  For $s\in \Z^+$,  $(1+\Delta)^{s}$ is a differential operator, while for nonintegral $s$, it is a $\pdo$ of order $2s.$  (Here $TLM$ is modeled on $H^{s'}$ sections of $\gamma^*TM$ with $s'\gg s.$)  The use of $\wgts$ is a standard analytic trick to impose regularity:  $X$ is at least $s-1$ times differentiable if 
$\Vert X\Vert_s <\infty.$  Note that $s=0$ recovers the $L^2$ metric.  From a physics point of view, we think of $s$ as a parameter we would like to set equal to infinity.  Since that is impossible, we want to extract information from these metrics that is independent of $s$.  

It was shown in \cite{Freed} that the Levi-Civita connection for the $H^s$ metric on loop groups has connection one-form taking values in $\pdo$s.  In \cite{MRTI}, this is extended to general loop spaces.  We only state the result for $s=1$.

 \begin{thm} \label{old1.6}
Let $\nabla^0$ be the Levi-Civita connection for the $L^2$ metric on $LM$, let $\nm$ be the Levi-Civita connection on $M$, and let $\Omega^M$ be its curvature two-form.
The $s=1$ Levi-Civita connection $ \nabla^1$ on $LM$ is given at the loop
$\gamma$ by
\begin{eqnarray}\label{s1}  \nabla^1_XY &=& \nabla^0_XY + \frac{1}{2}\wgti\left[
-\nabla^M_{\dg}(\Omega^M(X,\dg)Y) - \Omega^M(X,\dg)\nabla^M_{\dg} Y-\nabla^M_{\dg}(\Omega^M(Y,\dg)X)\right.\\
&&\qquad \left.  - \Omega^M(Y,\dg)\nabla^M_{\dg} X+\Omega^M(X,\nabla^M_{\dg} Y)\dg - \Omega^M(
\nabla^M_{\dg} X, Y)\dg\right].\nonumber
 \end{eqnarray}
\end{thm}

This is proven by examining the six-term formula as in (\ref{5one}).  As an operator on $Y$, $\nabla^0_XY$ is
zeroth order, while all other terms are order $-1$ or $-2$.  For example, in the term\\
$\wgti\nabla^M_{\dg}(\Omega^M(X,\dg)Y)$, $\wgti$ has order $-2$ and $\nabla^M_{\dg}(\Omega^M(X,\dg)Y)$ contains subterms of order $0$ and $1$
in $Y$.  Since orders add under composition of operators, the subterms
have orders as stated.  Although the appearance of the covariant derivative of the curvature 
in (\ref{s1}) is unwelcome,  the Levi-Civita connection 
 is explicit,  so that the symbol asymptotics of the curvature
$\Omega^1$ of $\nabla^1$ can be computed to any order \cite[Appendix]{MRTI}.  
Not surprisingly, $\Omega^1$ equals $\Omega^0$ plus a $\pdo$ of order at most $-1.$ 

This fits very well with Def. \ref{wcsdef} with $\nabla_0, \nabla_1$  the $L^2$ and $s=1$ Levi-Civita connections.  $\omega_1-\omega_0$ has strictly negative order, while $\Omega_t$ has its order $0$ term given by classical curvature expressions.  This makes the integrand 
$ {\rm res}^{\rm w} [(\omega_1-\omega_0)\wedge (\Omega_t)^{k-1}]$
in (\ref{5.22}) relatively straightforward to compute.

\begin{thm} \label{qq} Let dim$(M) = 2k-1$.  Fix a Riemannian metric on $M$ with curvature two-form $\Omega^M$, and fix $X_1,\ldots, X_{2k-1}\in T_\gamma LM.$  The
 k${}^{\it th}$ Wodzicki-Chern-Simons form $CS^W_{2k-1}(\nabla^1, \nabla^0)$
is given by
 \begin{eqnarray}\label{csg}
\lefteqn{CS^W_{2k-1}(\nabla^1,\nabla^0)(X_1,...,X_{2k-1}) }\\
&=&
\frac{4}{(2k-1)!} \sum_{\sigma} {\rm sgn}(\sigma) \int_{S^1}\tr[
 (\Omega^M(X_{\sigma(1)},\cdot)\dg)
 (\Omega^M)^{k-1}(X_{\sigma(2)},..X_{\sigma(2k-1)} )].\nonumber
\end{eqnarray}
\end{thm}

Here $\sigma$ is a permutation of $\{1,\ldots, 2k-1\}.$

In contrast, WCS classes for string theory, {\it i.e.}, ${\rm Maps}(\Sigma^2, M),$ are harder to compute, since $T{\rm Maps}(\Sigma^2, M)$ is modeled on sections of a bundle over $\Sigma^2.$  
Although the connection and curvature forms still take values in zeroth order $\pdo$s, the Wodzicki residue 
now involves $\sigma_{-2}$, which means that covariant derivatives of $\Omega^M$ appear in the analog of
Thm.~\ref{qq}.

\begin{rem}  These two theorems indicate that other traces on negative order $\pdo$s, classified in \cite{lesch-neira}, do not appear naturally in this context.  The zeroth order part of the connection and curvature forms come from the corresponding forms on $M$, so the connection and curvature forms of the $H^s$ metric on $LM$ take values in negative order $\pdo$s only if $M$ is flat.  Even though $LM$ is nontrivial in this case, the WCS forms will vanish.  The same remarks hold in string theory with flat target manifolds.
\end{rem}

For degree reasons, the form
$c_k^W(\Omega) = (k!)^{-1}\int_{S^*S^1} \tr \sigma_{-1}(\Omega^k)$ vanishes for dim$(M) = 2k-1$.
Thus the WCS class
\begin{equation}\label{wcsc}[CS^W_{2k-1}(\nabla^1,\nabla^0)]\in H^{2k-1}(LM,\C)
\end{equation}
is defined.  

If we use the $H^s$ Levi-Civita connection, we obtain
$CS^W_{2k-1}(\nabla^s,\nabla^0) = s\cdot CS^W_{2k-1}(\nabla^1,\nabla^0).$
 Therefore the $s$-independent information in this WCS class is given by setting $s=1$; in physics terminology, we have
 successfully regularized the WCS class.
 
In contrast to finite dimensions, $CS^W_3$ vanishes pointwise on 3-manifolds due to symmetries of the curvature tensor.  Thus we will consider 5-manifolds.

\subsection{ WCS classes and diffeomorphism groups}

In this subsection, we produce several infinite families of 5-manifolds $\bmk$ with $|\pi_1(\diff(\bmk))| = \infty.$

In general, information about $\diff(M)$ seems very difficult to come by.  For example, it is a theorem of Smale that $\diff(S^2) \sim O(3)$, where the tilde means homotopy equivalence, and a theorem of Hatcher that $\diff(S^3)\sim O(4).$  There is a good understanding of the homotopy type of the identity component of $\diff(M^2)$ and $\diff(M^3)$ for $M^3$ hyperbolic or  Seifert fibered.  In addition, one knows the stable homotopy groups of $\diff(S^n)$ modulo torsion.  These are all difficult results, and use very different techniques from ours.

Our main result Thm. \ref{bigthm} states that for every projective algebraic K\"ahler surface $M$, there is an infinite family $\bmk$ of 5-manifolds with $\pi_1(\diff(\bmk))$ infinite.  For specific K\"ahler surfaces, we can give more precise information.

To begin the construction of $\bmk$, recall that an $S^1$ action $a:S^1\times M\to M$ induces $a^L:M\to LM, a^D:S^1\to\diff(M)$ by
$a^L(m)(\theta) = a^D(\theta)(m) = a(\theta,m).$  Clearly $a^L$ and $a^D$ are closely related, and the following Lemma makes this explicit.  For notation, let $[a^L]$ denote $a^L_*[M]\in H_{{\rm dim}(M)}(M,\C).$

\begin{lem} \label{prop:two} Let dim$(M)=2k-1,$ and let $a_0, a_1:S^1\times M\to M$ be actions.

%  (i) If $\int_{[a^L_0]} CS^W_{\kk} \neq \int_{[a^L_1]} CS^W_{\kk}$, then $a_0$ and $a_1$
%    are not homotopic, and $[a^M_0]\neq [a^M_1]\in \pi_1(\maps(M)).$

(i)  Let $\alpha$ be a closed form on $LM$ of degree $2k-1$. If $\int_{[a^L_{0}]} \alpha \neq 
\int_{[a^L_{1}]} \alpha$, then 
 $[a^D_0]\neq [a^D_1]\in \pi_1(\diff(M),{\rm Id}).$

(ii)  If $\int_{[a_1^L]} CS^W_{\kk} \neq 0,$ then
  $\pi_1(\diff(M), {\rm Id})$  is infinite.

\end{lem}

Here and from now on, $CS^W_{\kk}$ denotes $CS^W_{\kk}(\nabla^1, \nabla^0).$
\medskip

\noindent {\it Sketch of proof.} (i) By  Stokes' theorem, $[a_{0}]\neq [a^L_{1}]\in H_{2k-1}(LM,\C).$
It follows that $a_0$ and $a_1$ are not homotopic, which implies that $[a^D_0]\neq [a^D_1]$.
See \cite{MRTII} for details.

(ii) 
Let $a_n$ be the $n^{\rm th}$ iterate of 
 $a_1$: $a_n(\theta,m) =
a_1(n\theta,m).$  
We outline a proof that 
 $\int_{[a^L_n]}CS^W_{\kk} =
n\int_{[a^L_1]}CS^W_{\kk}$.  For by (\ref{csg}), every term in $CS^W_{\kk}$ is of the
form $\ints\dot\gamma(\theta) f(\theta)$, where $f$ is a periodic function on the
circle.  Each loop $\gamma\in
a^L_1(M)$ corresponds to the loop $\gamma(n\cdot)\in a^L_n(M).$  Therefore 
$\ints\dot\gamma(\theta) f(\theta)$ is replaced by 
$$\ints \frac{d}{d\theta}\gamma(n\theta) f(n\theta)d\theta 
 = n\int_0^{2\pi} \dot\gamma(\theta)f(\theta)d\theta.$$
Thus $\int_{[a^L_n]}CS^W_{\kk} = n\int_{[a^L_1]}CS^W_{\kk}.$
 By (i), the $[a^L_n]\in 
\pi_1(\diff(M), {\rm Id})$
are all distinct.  
\hfill ${}\Box$
\bigskip

Lemma \ref{prop:two}(ii) gives us a strategy to produce 5-manifolds $M$ with infinite $\pi_1(\diff(M))$.  We 
want an $S^1$ action $a$ and a relatively computable metric on $M$.  If
$\int_{a_*[M]} CS^W_5\neq 0$, then $|\pi_1(\diff(M))| = \infty.$
From examples in the literature, especially \cite{gdsw}, it seems best to consider the total space of a circle bundle over a K\"ahler surface, as these spaces have an obvious $S^1$ action by rotating the circle 
fibers and 
carry Sasakian metrics closely related to the K\"ahler metric.

As pointed out to us by Alan Hatcher, it is not always the case that the fiber rotation is an element of infinite order in $\pi_1(\diff(M)).$  For the free action of $S^1$ on $S^5\subset \C^3$ given by $a(e^{i\theta},z) = e^{i\theta}z$ has quotient $M=\CP^2.$  The action is
via isometries for the standard metric on $S^5$, and so gives an element in $\pi_1({\rm Isom}(S^5))
= \pi_1(SO(6)) = \Z_2.$  Under the inclusion ${\rm Isom}(S^5)\to \diff(S^5)$, this element has order at most
two.  

In general,
let $(M^4, g, J,\omega)$ be an integral K\"ahler surface, i.e.~$J$ is the complex structure, $g$ is the 
K\"ahler metric, and the K\"ahler form is $\omega\in H^2(M,\Z).$  It follows from the Kodaira embedding theorem that $M$ is integral iff it is projective algebraic.
Fix $k\in \Z.$ 
As in geometric
quantization, we can construct a $S^1$-bundle $L_k\to M$ with connection 
$\be$ with curvature $d\be = k\omega.$  Let $\bmk$ be the total space of $L_k$.

$\bmk$ has a Sasakian structure; see \cite[\S4.5]{blair}, \cite{MRTII}, \cite[Lemma 1]{oneill} for details. The horizontal space of the connection is $\calH = {\rm Ker}(\be)$.  Define a metric $\bg$ on $\bmk$ by
$$\bg(\bx,\by) = g(\pi_*\bx, \pi_*\by) + \be(\bx)\be(\by).$$
Let $R, \overline R$ be the curvature tensors for $g, \overline g$, respectively.  By some careful computations relying heavily on the fact that $g$ is K\"ahler, we obtain:
\begin{lem}  \label{3.7}
\begin{eqnarray*}\bg(\br(X^L, Y^L)\zl,W^L) &=& \la R(X,Y)Z,W\ra + k^2[-\la JY,Z\ra\la JX,W\ra\\
&&\quad 
+\la JX,Z\ra\la JY,W\ra +2\la JX,Y\ra \la JZ,W\ra],\\
\bg(\br(\xl, \yl)\zl, \bxi) &=& 0,\\        %\bg(\br(\xl, \yl)\bxi, \zl) = 0\\
\bg(\br(\bxi,\xl)\yl,\bxi) &=& k^2\la X,Y\ra.
\end{eqnarray*}
\end{lem}

We want
to show
$$0\neq \int_{[a^L]}CS^W_5 = \int_{a_*[\bmk]}CS^W_5 = \int_{\bmk} a^*CS^W_5.$$
$a^*CS^W_5$ is a multiple $f$ of the volume form on $\bmk$.  If $\bxi$ is a unit length vertical vector and
$(e_2, Je_2, e_3, Je_3)$ is a positively oriented orthonormal frame on $M$, then
$f = CS^W_{5,\gamma}(\bxi, e_2, Je_2, e_3, Je_3).$  A long computation in \cite{MRTII} using (\ref{csg}) and the previous Lemma gives
\begin{eqnarray}\label{bsum}\lefteqn{CS^W_{5,\gamma}(\bxi, e_2, Je_2, e_3, Je_3) }\nonumber\\
&=& \frac{k^2}{30} \left\{ 32\pi^2 p_1(R)(e_2, Je_2, e_3, Je_3)
+ 32k^2[3 R(e_2, Je_2, e_3, Je_3) -R(e_2, e_3, e_2, e_3)   \right.\nonumber\\
&&\quad \left. %+  R(e_2, Je_3, e_2, Je_3)
-  R(e_2, Je_3, e_2, Je_3)  + R(e_2, Je_2, e_2, Je_2)+ R(e_3, Je_3, e_3, Je_3)]\right.\\
 &&\quad \left.+ 192k^4 \right\},\nonumber
\end{eqnarray}
 where $p_1(R)$ is the first Pontrjagin form.  
This leads to a crucial estimate. Set 
$$|R|_\infty = \max_E\{|R(e_i, e_j, e_k, e_\ell)\},$$ 
where $E$ is the set of  orthonormal frames at all points of $M$.  

\begin{prop} \label{39}
$\int_{\bmk} CS^W_5 >0$ if
$$k^2\left(96\pi^2 \sigma(M) -224k^2|R|_\infty  \vol(M) + 192k^4 
\cdot \vol(M)\right) >0.$$
\end{prop}
 
 Here $\sigma(M) = \frac{1}{3}\int_M p_1(R)$ is the signature of $M$.  Since the $k^4$ term will dominate 
 for $k\gg 0$, we get
 \begin{thm}\label{bigthm}  Let $(M^4, J, g, \omega)$ be a compact integral K\"ahler surface, and let $\bmk$ be the circle bundle associated to $k[\omega]\in H^2(M, \Z)$ for $k\in \Z.$  Then
the loop of diffeomorphisms of $\bmk$ given by rotation in the circle fiber gives an element of infinite order in $\pi_1(\diff(\bmk))$ for
 $|k| \gg 0$.  This loop is also an element of infinite order in 
 $\pi_1({\rm Isom}(M)).$
\end{thm}
The last statement follows as in the $S^5$ example. 

We note that these results tell us nothing if $k=0$, i.e. for $\overline M_0 = M\times S^1.$  One would think this is the easiest case, but our methods fail here.

For specific K\"ahler metrics, we can give more precise results using (\ref{bsum}).
For notation, on
$M = \CP^1\times\CP^1$, let $\omega_1, \omega_2$ be the standard 
K\"ahler form on each $\CP^1$ with sectional curvature $1$.  For  $a,b\in\Z^+$, let $\omega = a\omega_1+b\omega_2$ be an integral 
K\"ahler form on $M$, and
let $\mabk$ be the total space of the line bundle associated to $k\omega.$ 
For $M$ a projective algebraic K3 surface, recall that 
$H^2(M) \simeq \Z^{22}.$ Fix an integral K\"ahler class
$[\omega] = [\omega_1,\ldots,\omega_{22}]$ in the obvious notation. Take $a_1,\ldots a_{22}\in \Z^+.$
 For $k\in \Z\setminus\{0\}$, let $\bka$ be the total space of the line bundle associated to 
$k\sum_{i=1}^{22} a_i\omega_i.$

\begin{thm}\label{lastthm} (i) $\pi_1(\diff(\overline{T}^4_k))$ is infinite for $k\neq 0.$

(ii) $\pi_1(\diff(\overline{\CP^2_k}))$ is infinite for $k\neq 0, \pm 1.$

(iii) For $a, b\in \Z^+, k \neq 0$,  $\pi_1(\diff(\mabk))$ is infinite.

(iv) Let $M$ be a projective algebraic K3 surface.  $\pi_1(\diff(\bka))$ is infinite for $k\neq 0.$  

(iv)  There are infinitely many values of $k_1, k_2, k_3, k_4, a,b, \vec a$ such that 
$\overline{T^4_{k_1}}, \overline{\CP^2_{k_2}},
\overline M_{k_3(a, b)}, \overline M_{k_4\vec a}$ are mutually nonhomeomorphic.
\end{thm}

(i) follows immediately from Prop.~\ref{39}, since $|R|_\infty = 0$ on the flat torus and $\sigma(T^4) = 0.$
For (ii), (\ref{bsum}) vanishes only for $k = 0,\pm 1$, as it must; this gives us confidence that the constants in 
(\ref{bsum}) are correct.  (iii) uses the Ricci flat metric on a K3 surface and
the decomposition of $\Lambda^2(M)$ into selfdual and anti-selfdual forms. (iv)  follows from Gysin sequence computations of the cohomology of these spaces.  Details are in \cite{MRTII}.

\bibliographystyle{amsplain}
\bibliography{Paper}

\end{document}